\documentclass[11pt]{amsart}

\usepackage[english]{babel}
\usepackage[utf8x]{inputenc}
\usepackage[T1]{fontenc}

\usepackage[margin=1.5in]{geometry}

\usepackage{amsmath}
\usepackage{amstext}
\usepackage{amsfonts}
\usepackage{amssymb}
\usepackage{amsthm}
\usepackage{amsrefs}
\usepackage{mathtools}
\usepackage{dsfont}
\usepackage{color}
\usepackage{graphicx}
\usepackage[colorinlistoftodos]{todonotes}
\usepackage[colorlinks=true, allcolors=blue]{hyperref}
\usepackage{float}
\usepackage{mathrsfs}
\allowdisplaybreaks

\newtheorem*{cor*}{Corollary}
\newtheorem*{thm*}{Theorem}
\newtheorem*{lem*}{Lemma}

\newtheorem*{prop*}{Proposition}
\newtheorem*{rmk*}{Remark}

\theoremstyle{definition}

\newtheorem*{defn*}{Definition}

\theoremstyle{remark}

\newcommand{\Rad}{\operatorname{Rad}}

\usepackage[nodisplayskipstretch]{setspace}

\allowdisplaybreaks

\makeatletter
\renewcommand{\@settitle}{%
  \begin{center}%
    \baselineskip14pt\relax
    \bfseries
    \@title
  \end{center}%
}
\makeatother

\title[]{\large Groups with rapid decay and trivial amenable radical are C$^*$-simple}

\usepackage{etoolbox}

\makeatletter
\patchcmd{\@setauthors}
  {\MakeUppercase{\authors}}
  {\authors}
  {}
  {\ClassWarning{amsart}{Could not patch author capitalization}}
\makeatother

\author[C. Barnett]{\normalsize Caleb Barnett}
\address{Caleb Barnett, University of Illinois at Chicago\\ USA}
\email{cbarn@uic.edu}

\author[J. Gondek]{Joseph Gondek}
\address{Joseph Gondek, University of Oxford\\ UK}
\email{joseph.gondek@stx.ox.ac.uk}

\author[M. Kalantar]{Mehrdad Kalantar}
\address{Mehrdad Kalantar\\ University of Oxford\\ UK}
\email{mehrdad.kalantar@maths.ox.ac.uk}

\begin{document}

\maketitle

\let\thefootnote\relax\footnotetext{
For the purpose of open access, the authors have applied a CC-BY license to any author accepted manuscript arising from this submission.}

In \cite{agkp26}*{Question 1}, Alam, Pham, and the second and third named authors asked whether a finitely generated group with rapid decay and trivial amenable radical is C$^*$-simple. Recently, the first author communicated to the second and third authors the existence of a Lean-verified proof of a positive answer to this question, generated by GPT-6 Astra. An online meeting between the authors was arranged, where the first author would present the LLM-generated proof to the second and third authors. The first author began the presentation with an elementary observation: if an element $g\in\Gamma$ lies outside of the amenable radical, then a finite set of conjugates $F\subset \textup{Cl}(g)$ generate a non-amenable subgroup of $\Gamma$. We immediately understood that this basic insight was missed in \cite{agkp26}; it implies exponential co-growth of the centralizer of $g$, which is enough to quickly answer \cite{agkp26}*{Question 1}. The three authors agreed to end the presentation at that point, and write this note, for which no LLM was used: the second and third authors have not read or verified the contents of the LLM-generated proof. The first author intends to post the LLM-generated proof in a separate, independent note.

\medskip

Let $B_n$ denote the ball of radius $n$ centered at the neutral element. Following the proof of \cite[Theorem A]{agkp26}, the observation mentioned above implies:

\begin{lem*} 
Let $F\subset \Gamma\backslash \textup{Rad}(\Gamma)$ be finite. There are subsets $A_n\subset B_n$ of \hbox{exponential} growth such that the maps $
A_n\ni a\mapsto aga^{-1}$
are injective, for all $g\in F$.
\end{lem*}

\begin{proof}
We proceed by induction. If $g\in \Gamma\backslash \textup{Rad}(\Gamma)$, then find elements $t_1, ..., t_m\in \Gamma$ such that the subgroup $H = \langle t_kgt_k^{-1}: 1\leq k \leq m\rangle\leq\Gamma$ is nonamenable. Let $B_k(H)$ denote the ball of radius $k$ in $H$, with respect to this generating set, and consider the well-defined injective map
\begin{align*}
\mathcal{Q}: H/Z(H)\ni hZ(H)\mapsto (ht_1gt_1^{-1}h^{-1}, ..., ht_mgt_m^{-1}h^{-1}).
\end{align*}
Because $H/Z(H)$ is nonamenable, it has exponential growth, and therefore $\mathcal{Q}(B_k(H))\subset \prod_{k=1}^m H$ has exponential growth. It follows (from e.g., the pigeonhole principle), that there is $1\leq k_0\leq m$ such that the subset $\{st_{k_0}gt_{k_0}^{-1}s^{-1}: s\in B_n(H)\}$ has exponential growth in $H$, and therefore $\{tgt^{-1}: t\in B_n\}$ has exponential growth in $\Gamma$. This proves the base case $|F| = 1$.

Now suppose the theorem is proven for subsets $F\subset \Gamma$ of cardinality $|F| < m$, and consider $m$
distinct elements $g_1, ..., g_m\in \Gamma\backslash \textup{Rad}(\Gamma).$ Let $A_n\subset B_n$ be a sequence of subsets with exponential growth such that the maps
$A_n\ni a\mapsto ag_ka^{-1}$
are injective, for all $k < m$. Consider the number of fibers $\iota(n)$ for the maps $\pi_n: A_n\ni a\mapsto aC_\Gamma(g_m)\in \Gamma/C_{\Gamma}(g_m)$. 
If there is a strictly increasing sequence $(n_k)_{k\in\mathbb{N}}$ with $\iota(n_k) \sim \alpha^{n_k}$ for some $\alpha > 1$, we are done. 
If not, then for each $n$ there is a fiber, say $A_n\supset A_n' = \pi_n^{-1}(\{x_n\})$ for $x_n\in A_n$, of exponential growth. We will use $A_n'$ to construct a new subset $A_n''\subset B_{3n}$ with the property that the maps $A_n''\ni a\mapsto aC_\Gamma(g_k)$ are injective, for every $k = 1, ..., m$, and this will complete the proof. We will do this with a diagonalization argument. Let $f(n) = \min\{|A_n'|, |q_m(B_n)|\}$ (where $q_m: \Gamma\to \Gamma/C_\Gamma(g_m)$ is the coset map), and enumerate $A_n' = \{a_1, ..., a_{P(n)}\}$, $q_m(B_n) = \{b_1, ..., b_{Q(n)}\}$. The functions $P(n)$, $Q(n)$, and $f(n)$ have exponential growth.
Find $N\in\mathbb{N}$ with $f(n)> 2m$ for all $n > N$, and now fix such $n$. Consider the $Q(n)$ subsets $b_1x_n^{-1}A'_n, b_2x_n^{-1}A'_n, ..., b_{Q(n)}x_{n}^{-1}A'_n$. We construct $A_n''$ inductively: suppose that $\ell-1$ elements $y_i\in b_ix_n^{-1}A_n'$ ($i <\ell < f(n)/2m$) have been chosen with distinct $C_{\Gamma}(g_k)$-cosets for all $k = 1, ..., m$. We choose the $\ell$th element of $A_n''$ by removing all elements from $b_{\ell}x_{n}^{-1}A_n''$ which share a $C_{\Gamma}(g_k)$-coset with some $y_i$, for $i < \ell$ and $k < m$ (there are at most $\ell m$ such elements), and choosing an element from the remaining set (such an element exists by the choice of $N$).
Because $n > N$, this process may be continued $f(n)/2m$ times, so that $A_n''$ is a set of cardinality $f(n)/2m$ 
for which all $m$ conjugation maps are injective. This completes step $|F| = m$.
\end{proof}
\begin{thm*}\label{mainthm}
If $\Gamma$ has \textup{RD} and $\textup{Rad}(\Gamma) = \{e\}$, then $\Gamma$ is $C^*$-simple.
\end{thm*}
\begin{proof}
We show that if $g\not\in \Rad(\Gamma)$, then for the Furstenberg boundary action $\Gamma\curvearrowright \partial_F\Gamma$, we have $\text{Fix}(g) = \emptyset$. If not, then there is a finite subset $F\subset \textup{Cl}(g)$ such that $\bigcup_{h\in F}\textup{Fix}(h) = \partial_F\Gamma$. Let $(A_n)_{n\in\mathbb{N}}$ satisfy the hypotheses of the Lemma for the set $F$, let $c = \max\limits_{h\in F}|h|$, find $N\in\mathbb{N}$ with $\rho(2N+c)/|A_N|^{1/2} < 1/|F|$. Let $\mu\in \textup{Prob}(\Gamma)$ be the uniform measure on $A_N$, and let $\nu\in \textup{Prob}(\partial_F\Gamma)$ be a $\mu$-stationary measure. Arguing as in the proof of \cite{agkp26}*{Theorem A}, we have
\begin{align*}
\nu(\textup{Fix}(h)) \leq \frac{1}{|A_N|}\|\sum_{a\in A_N}\lambda_{aha^{-1}}\|\leq \frac{\rho(2N + c)}{|A_N|}\|\sum_{a\in A_N}\delta_{aha^{-1}}\|_2 = \frac{\rho(2N + c)}{|A_N|^{1/2}} < \frac{1}{|F|},
\end{align*}
so $1 = \nu(X)\leq \sum_{h\in F}\nu(\textup{Fix}(h)) < \sum_{h\in F}|F|^{-1} = 1$, an impossibility.
\end{proof}

\begin{rmk*}
The above holds also for groups with \textup{SRD} in the sense of \cite{agkp26}.
\end{rmk*}

\nocite{*}
\bibliographystyle{plain}
\bibliography{bgk}

\end{document}